\documentclass{article}
\usepackage{amsmath,latexsym,amssymb,amsthm,enumerate,amscd}
\theoremstyle{plain}
\newtheorem{theorem}{Theorem}[section]

\newtheorem{corollary}[theorem]{Corollary}
\newtheorem{lemma}[theorem]{Lemma}
\theoremstyle{definition}
\newtheorem{definition}[theorem]{Definition}

\newtheorem{remark}[theorem]{Remark}
\newtheorem{example}[theorem]{Example}

\newtheorem{thm}{Theorem}

\theoremstyle{plain}

\theoremstyle{definition}

\newtheorem*{ack}{Acknowledgment}
\numberwithin{equation}{section}
\numberwithin{table}{section} %TTT
\def\AFG{\mathrm{AFG}}

\def\cha{\mathrm{char}\ }

\def\Hilb{\mathrm{Hilb}}

\providecommand{\bysame}{\makebox[3em]{\hrulefill}\thinspace}

\def\<{\left<}
\def\>{\right>}

\def\Ker{\mathrm{Ker}}

\def\Sym{\mathrm{Sym}}

\def\ns{\footnotesize \it}

\def\rk{\mathrm{rank}}

\def\max{\mathrm{max}}
\title{Pairs of commuting nilpotent matrices, and Hilbert function}
\author{Roberta Basili\footnote{The first author was supported partially by
University of Perugia and the Italian G.N.S.A.G.A. during her visit of summer 2003 to the
Mathematics Department of Northeastern University; and by an NSF grant 
 of J. Weyman during her visit in summer 2006.}\\[.05in]
{\ns  Via dei Ciclamini 2B, 06126 Perugia, Italy.
}\\[.2in]
Anthony Iarrobino\\[.05in]
{\ns Department of Mathematics,  Northeastern University, Boston, MA 02115, USA.
}\\[.2in]}

\date{February 6, 2008}
\begin{document}
\maketitle
\begin{abstract} Let $K$ be an infinite field. There has been substantial recent study of the family 
 $\mathcal H(n,K)$ of pairs of commuting nilpotent $n\times n$ matrices, relating this family to the fibre $H^{[n]}$ of the punctual Hilbert scheme 
$A^{[n]}=\Hilb^n(\mathbb A^2)$ over the point $np$ of the symmetric product $\Sym^n(\mathbb A^2)$, where $ p$ is a point of 
the affine plane $ \mathbb A^2$ \cite{Bar,Bas1,Prem}. In this study a pair of commuting nilpotent matrices $(A,B)$ is related to an Artinian algebra
 $K[A,B]$. There has also been substantial study of the stratification of the local punctual Hilbert scheme $H^{[n]}$ 
by the Hilbert function \cite{Br,I1,ElS,Ch,Go,Gu,Hui,KW,IY,Yama,Yamb}. However these studies have been hitherto separate. \par 
We first determine the stable partitions: i.e. those for which $P$ itself is the
partition $Q(P)$ of a generic nilpotent element of the centralizer of the Jordan nilpotent matrix $J_P$. We then explore the relation between $\mathcal H(n,K)$ and its stratification by the Hilbert function
of $K[A,B]$. Suppose that $\dim _K K[A,B]=n$, and that $K$ is algebraically closed of characteristic 0 or large enough $p$.  We show that a generic element of the pencil $A+\lambda B, \lambda\in K$ has Jordan partition the maximum
partition $P(H)$ whose diagonal lengths are the Hilbert function $H$ of $K[A,B]$. 
These results were announced in the talk notes \cite{I4}, and have been used by T. Ko{\v{s}}ir and
P.~Oblak in their proof that $Q(P)$ is itself stable \cite{KO}.  
\end{abstract}
\section{Pairs of commuting nilpotent matrices.}\label{sec1}
\subsection{Introduction}
We assume throughout Section \ref{sec1} that $K$ is an infinite field. Further assumptions on $K$, when needed, will be explicitly stated in each result. Given $B=J_P\in 
M_n(K)$,
 a nilpotent $n\times n$ matrix in Jordan form corresponding to the
partition $P$ of $n$, we denote by $\mathcal C_B$ the centralizer of $B$,
\begin{equation}
\mathcal C_B=\{ A\in M_n(K)\mid [A,B]=0\},
\end{equation} 
and by $\mathcal N_B$ the set of nilpotent elements of $\mathcal C_B$. They each have a natural scheme structure. It
 is well known that $\mathcal N_B$ is an
irreducible algebraic variety (\cite[Lemma 2.3]{Bas2}, see also Lemma \ref{nilpBirredlem} below). Thus there is a Jordan partition that we will denote
$Q(P)$ of a generic matrix $A\in \mathcal N_B$. Several have studied the problem of determining $Q(P)$ 
given $P$ \cite{Oblak,Oblak2,KO, Pan}.  We here first determine the ``stable'' partitions $P$ under $P\to Q(P)$ -- that is, those $P$ for which $Q(P)=P$ -- using results from \cite{Bas2} (see Theorem \ref{stablecor} below).
\begin{thm}$P$ is stable if and only if the parts of $P$ differ pairwise by at least two.
\end{thm}\par\noindent
We next in Section \ref{hilbsec} consider a pair of commuting $n\times n$ nilpotent matrices $(A,B)$ such that $\dim_K K[A,B]=n$.
The ring $\mathcal A=K[A,B]\cong K\{x,y\}/I_{A,B}$ has a Hilbert function $H=H(\mathcal A )$
satisfying 
\begin{equation*}
H=(1,2,\ldots ,\nu ,t_{\nu},\ldots , t_j,0)  \text{ where } \nu\ge t_{\nu}\ge \cdots \ge t_j>0,
\end{equation*}
where $j$ is the socle degree of $H$. We denote by
$P(H)$ the dual partition to the partition of $n$ given by $H$: thus, the entries of $P(H)$ are the lengths of the
rows of the bar graph of $H$ (Definition \ref{diagdef}). We denote by $\mathcal U_B\subset \mathcal N_B$ the dense subset
$\{ A\in {\mathcal N}_B\mid \dim_K K[A,B]=n\}$. Considering an element of the pencil $C_\lambda =A+\lambda B, \lambda \in K$, and the multiplication endomorphism $\times (
A+\lambda B)$ 
it induces on $K[A,B]$, we have (Theorems \ref{genericlem} and \ref{decreasebthm})
\begin{thm}\begin{enumerate}[A.]
\item Suppose $A\in \mathcal U_B$, let $H=H(K[A,B])$ of socle degree $j$, and let $K$ be an algebraically closed field 
with $\cha K=0$ or $\cha K>j$.
 Then for generic $\lambda\in \mathbb
P^1$ the Jordan block sizes of the action of $A+\lambda B$ on $K[A,B]$ are given by the parts of $P(H)$.
\item Assume further $\cha K=0$ or $\cha K>n$. Then the partition $Q(P)$ satisfies
\begin{equation*}
Q(P)=\max_{A\in \mathcal U_B} P(H(K[A,B])),
\end{equation*}
and has decreasing parts.
\end{enumerate}
\end{thm}
These results were announced in the talk notes \cite{I4}, and have been used by T. Ko{\v{s}}ir and
P.~Oblak in their proof that $Q(P)$ is itself stable \cite{KO}. We state their result in Theorem \ref{Gorensteinthm}.
\subsection{Stable partitions $P$}
We denote by $P=(p_1,\ldots ,p_t), p_1\ge
\cdots
\ge p_t\ge 1$ a partition $P$ with t parts (so the Jordan nilpotent matrix of partition P has rank $n-t$); we let $n(i)=\# $ parts of $P$ at least
$i$. Then the dual partition $\hat{P}$ (switch rows and columns in the Ferrers graph of $P$) satisfies $\hat{P}=(n(1),n(2) \ldots )$.
The following lemma is well known and motivates Definition \ref{decompdef}.
\begin{lemma}[Jordan blocks of ${J_P}^i$]\label{partpowlem}
Consider the $n\times n$ Jordan matrix $J_P$ of partition $P$. Then
\begin{enumerate}[i.]
\item\label{partpowlemi} For $P=[n]$, a single block, the partition of $(J_P)^i$ for $i\le n$ is
the unique partition of $n$ having
$i$ parts of sizes differing by at most 1. For $P=[n]$ and $i>n$ the partition of
$(J_P)^i$ has $n$ parts of size 1.
\item\label{partpowlemii}
For an arbitrary $P$, the Jordan partition of $(J_P)^i$ is the union of the partitions for
$(J_{[p_k]})^i, k=1,\cdots ,t$.
\item The rank of $(J_P)^i$ satisfies
\begin{equation}\label{rankeq}
\text { rank } (J_P)^i=n-\left( n(1)+\cdots + n(i)\right).
\end{equation}
\item\label{ranktoP} Let $A$ be nilpotent $n\times n$. The difference sequence $\Delta$ of $(n,rk(A^1),rk(A^2), \ldots )$
is the dual partition $\hat {P_A}$ to $P_A$, the Jordan partition of $A$.
\end{enumerate}
\end{lemma}
\begin{proof} Here \eqref{ranktoP} follows from \eqref{rankeq}.
\end{proof}\par
\begin{example} For $P=[7]$, $(J_P)^2$ has blocks $(4,3)$, $(J_P)^3$ has blocks $(3,2,2)$,
$(J_P)^4$ has blocks $(2,2,2,1)$.  
\end{example}
\begin{definition}\label{decompdef}
We term a partition $P$ whose largest and smallest part differ by at most one, a ``string''. Such a $P$ is termed
``almost rectangular'' in \cite{KO} , since its Ferrer's graph (Definition~\ref{diagdef}) is obtained by removing a portion of the last column from that
of a rectangular partition.
Each partition $P$ is the union $P=P(1)\cup\ldots \cup P(r)$ of strings $P(i)$.
We let $r_P$ be the minimum number $r$ of subpartitions $P(i)$ in any such decomposition of $P$.
\end{definition}
\begin{example} For $P=(5,4,4,3,2)$ we may subdivide $P=(5,4,4)\cup (3,2)$, which gives
$r_P=2$. For $P=(8,7,7,7,5,5,4,2,1), r_P=3$. The subdivision into $r_P$ strings need not
be unique: for $P=(5,4,3,2,1)= (5,4)\cup (3,2)\cup (1)$ or $(5,4)\cup ( 3)\cup (2,1)$, with
$r_P=3$.  
\end{example}
Recall that we denote by ${\mathcal N}_B$ the set of nilpotent elements of the centralizer $\mathcal C_B$, endowed with its natural structure as a scheme \cite{Bo}.
 R. Basili showed in \cite[Lemma 2.3]{Bas2} based on \cite{TA}, that the nilpotent commutator $\mathcal N_B$ of
a nilpotent matrix $B$ is an irreducible variety. For completeness we include a proof suggested by the referee of the following more general statement.
\begin{lemma}\label{nilpBirredlem} If $\mathfrak A$ is a finite dimensional algebra over an infinite field $K$, then the scheme $\mathcal N(\mathfrak A)$ of nilpotent elements of $\mathfrak A$ is an irreducible variety.
\end{lemma}
\begin{proof} Since irreducibility is a geometric property, we may make a base change and assume that $K$ is algebraically closed. Let $\mathfrak J$ be 
the Jacobson radical of $\mathfrak A$. Then Wedderburn's theorem yields a semisimple subalgebra $\mathfrak L\subset \mathfrak A$ such that ${\mathfrak A}=
{\mathfrak L}\oplus \mathfrak J$, an internal direct sum as vector spaces; and the restriction of the natural projection $p: {\mathfrak A} \to {\mathfrak A}/\mathfrak J$ gives
an isomorphism
\begin{equation} p_{\mid \mathfrak L}: {\mathfrak L}\to {\mathfrak A}/\mathfrak J
\end{equation}
Now $\mathcal J$ is a nilpotent ideal, and ${\mathfrak L}\cong {\mathfrak A}/\mathfrak J$ is a (split) semisimple algebra over $K$. Thus $\mathfrak L$ is a direct
product of matrix algebras ${\mathrm{Mat}}_{r_u} (K)$ for certain $r_u$, by another theorem of Wedderburn. It is well known that the set of nilpotent elements
$\mathcal N({\mathfrak L})$ is irreducible, since the unit group $\mathfrak {L}^\ast$ of $\mathfrak L$ is a connected algebraic group, and has a dense orbit
on $\mathcal {\mathfrak L}$.\par
Now, an element $(\ell, j) \in {\mathfrak L}\oplus {\mathfrak J} = \mathfrak A$ is nilpotent when $\ell$ is nilpotent, and $\mathfrak J$ is an ideal, so as a 
variety is just a copy of an affine space, so is irreducible.  Thus the nilpotent commutator
$\mathcal N({\mathfrak L})\times \mathfrak J$ is the product of irreducible varieties, hence is irreducible.
\end{proof}
\par
Note that the proof that $\mathcal N_B$ is irreducible given in \cite[Lemma 2.3]{Bas2} is essentially an application of the above proof to 
the special case ${\mathfrak A}={\mathcal C}_B$, the centralizer of $B$.
 R. Basili uses there a specific parametrization of $\mathcal N_B$: certain matrices $\tilde{\mathcal A}_{u,u}$ 
appearing there for $\mathcal C_B$, whose nilpotence defines $\mathcal N_B$, are the elements of the 
matrix algebras $\mathrm{Mat}_{r_u}(K)$ in the above proof. Here $r_u$ is the multiplicity of the $u-$th distinct part of $P$.\smallskip
\par
Recall that, given a partition $P$, we denote by $B=J_P$ the Jordan nilpotent matrix of partition $P$.
It follows from Lemma \ref{nilpBirredlem}
that there is a unique partition $Q(P)$ that occurs for a generic element of $\mathcal N_B$. 
\par 
We recall the natural majorization partial order on the partitions $P$ (we assume $p_1\ge p_2\ge \cdots \ge p_t$).
\begin{equation}\label{poparteq}
P\ge P' \text{ if and only if for each } i, \sum_{1\le u\le i} p_i\ge \sum_{1\le
u\le i} p'_i.
\end{equation}
From Lemma \ref{partpowlem} it is easy to see that 
\begin{equation}
P\ge P' \Leftrightarrow  \forall i, \rk ({J_P}^i)\ge \rk ({J_{P'}}^i).
\end{equation}
We let $O_P$ denote the $Gl(n)$ orbit of $J_P$.
 We have \cite{H}
\begin{equation}  \overline{O_P}\supset O_{P'}\Leftrightarrow  P\ge P'.
\end{equation}
\begin{lemma}\label{QPlem}
The partition $Q(P)$ determined by the Jordan block sizes of a generic
element of
$\mathcal N_B$ satisfies $Q(P)\ge P_A$ for each $ A\in \mathcal N_B$.
\end{lemma}
\begin{proof} This follows from the irreducibility of $\mathcal N_B$, from \eqref{rankeq}, and the semicontinuity
of the ranks of powers of $A$.
\end{proof}\par
Before the present work was announced \cite{I4}, there were several results known about $Q(P)$. 
\begin{theorem}\label{partsthm}\cite[Proposition 2.4]{Bas2}
The rank of a generic element $A\in \mathcal N_B$ is $n-r_P$.
Equivalently, the partition $Q(P)$ has $r_P$ parts.
\end{theorem}\noindent
Also, P. Oblak had determined the ``index'' or largest part of $Q(P)$ using graph theory \cite{Oblak}. We subsequently have given
another proof of Oblak's result (see \cite{BI1}).\par
We use the notation $\mid P\mid =n$, the integer partitioned by $P$.
\begin{definition}  
Let $\mathcal P=(P_1,\ldots, P_{r_P})$
be a decomposition of $P$ into $r_P$ non-overlapping strings:
\begin{equation}
\bigcup_i P_i=P, \text { and } P_i\cap P_j=\emptyset \text { if } i\not=j. 
\end{equation}
Given such a decomposition $\mathcal P$ of $P$, we
denote by $\tilde {\mathcal P}$ the partition $(\mid P_1\mid, \ldots ,\mid P_{r_P}\mid)$, rearranged in decreasing order. 
\end{definition}
For $P=(3,3,3,2,2,1)$ two such decompositions into strings are $\mathcal P=((3,3,3),(2,2,1)))$ and $\mathcal P'=((3,3,3,2,2),(1))$.
We have $\tilde {\mathcal P}=(9,5)$ and $\tilde {\mathcal P'}=(13,1)$. Here $r_P=2$.

\par\smallskip
\begin{lemma}\label{qbiggerlem} Suppose that the partition $P$ of $n$ contains two parts that are equal, or that
differ by one. Then $Q(P)>P$. 
\end{lemma}
\begin{proof} Assume that $P$ has two parts that are the same
or that differ by one. Choose a decomposition $\mathcal P$ into $r_P$ strings $P_1,\ldots P_{r_P}$. We claim that some nilpotent matrix $\tilde{B}$ of partition $\tilde {\mathcal P}$ commutes with $J_P$. To show this, we may first reduce to the case that $P$ has two parts, which differ by $0$ or $1$. 
We have by Lemma \ref{partpowlem}\ref{partpowlemi} that the partition of $A=(J_{n})^2$ is $P$. Then $gAg^{-1}=J_P$ for some $g\in \mathrm{Gl}_n(K)$, so the nilpotent matrix
$g J_{n }g^{-1}$ centralizes $J_P$ and and has partition $P'=\mid P\mid$. This proves the claim.
Also
$P'$ is different from
$P$ since at least one string of $P$ has length greater than one, and $P'>P$. We have
by Lemma~\ref{QPlem} that $Q(P)\ge P'$, so
$Q(P)>P$.
\end{proof}\par
Note that when $P=(2,2)$, then $P'=(4)$, and $J_{P'}$ does not itself commute with $J_{P}$. \par\smallskip
We now determine the ``stable'' partitions $P$, for which $Q(P)=P$.
We need the following result of R. Basili. Given a partition $P$, let $s_P$ be the length of the longest string in $P$,
\begin{equation*}\label{sPeq}
s_P = \max \{i\mid \exists k\mid  (p_k\ge p_{k+1}\ge\cdots \ge p_{k+i-1})\subset P\text { and } p_k-p_{k+i-1}\le 1\}.
\end{equation*}
 For $P=(5,4,4,3,2)$ we have $s_P=3$. Note that $s_P=1$ iff the parts of $P$ differ by at least two.\par
The next theorem shows that the Jordan partition $P_{A^{s_P}}$ of the $s_P$ power of any element $A\in \mathcal N_B$ 
satisfies  $P_{A^{s_P}}\le P=P_B$.
\begin{theorem}\label{bas3.5thm} \cite[Proposition 3.5]{Bas2}
Let $B\cong J_P$ be nilpotent of Jordan partition $P$, and let $A\in \mathcal N_B$, the
nilpotent commutator of
$B$. Then
\begin{equation}\label{rank2eq}
rank ( A^{s_P})^m\le rank (B^m).
\end{equation}
\end{theorem}\noindent
\begin{theorem}\label{qpthm} Suppose that $P$ has a decomposition $\mathcal P$ into $r_P$ strings, each of length $s_P$.
Then $Q(P)=\tilde {\mathcal P}$.
\end{theorem}
\begin{proof} The assumption is equivalent to there being a unique decomposition of $P$ into $r_P$ strings, and also that these strings have equal length. Let $B=J_P$ and $s=s_P$. The proof of Lemma \ref{qbiggerlem} 
 implies there exists $\tilde {B}\in N_B$ of
 partition $\tilde {\mathcal P}$. By Lemma \ref{partpowlem} \eqref{partpowlemii} ${\tilde {B}}^s$ has Jordan blocks given by the partition $P$, so for $A=\tilde{B}$ there is equality in
\eqref{rank2eq} of Theorem~\ref{bas3.5thm}. Hence, by semicontinuity of rank, for an open dense subset of $A\in\mathcal N_B$, the Jordan partition of ${A^s}$ is $P$. By the unicity of the 
decomposition of $P$ into $r_P$ strings, and using Theorem~\ref{partsthm} we conclude that the partition of $A$ is constant on that subset, implying $Q(P)=\tilde {\mathcal P}$.   
\end{proof}\par
Thus, for $P=(5,4,2,2)$, we have $Q(P)=(9,4)$; for $P=(8,7,7,5,5,4,2,2,2), Q(P)=(22,14,6)$.\par
Given a positive integer $c$, we denote by $cP$ the partition obtained by repeating $c$ times each part of $P$. For $P=(3,1,1)$, $2P=(3,3,1,1,1,1)$.
\begin{theorem}[Stable partitions]\label{stablecor} The following are equivalent.
\begin{enumerate}[i.]
\item The parts of the partition $P$ differ pairwise by at least two:
 \begin{equation} 
 P=(p_1,\ldots ,p_t), p_1\ge \ldots \ge p_t, \text { and for } 1\le u\le t-1 , p_u-p_{u+1}\ge 2.
\end{equation}
\item  $Q(P)=P;$ 
\item For some positive integer $c$,
$Q(cP)=(cp_1,cp_2,\ldots ,cp_t)$. 
\end{enumerate}
\end{theorem}
\begin{proof} Theorem \ref{qpthm} shows (i)$\Rightarrow$ (ii), and (i)$\Rightarrow$ (iii). (ii)$\Rightarrow $ (i) is from Lemma \ref{qbiggerlem}. To show (iii) $\Rightarrow$ (i), suppose, by 
way of contradiction, that $P$ has two parts that differ by at most one: it suffices to assume $P=(p_1\ge p_2)$. For any $c\ge 1$ one can use Lemma \ref{partpowlem}\eqref{partpowlemi} in a way analogous
to the proof of Lemma \ref{qbiggerlem} to find a nilpotent matrix $A$ whose $2c$-th power is $J_{cP}$ and whose partition has the single part $c(p_1+p_2)$; but then $Q(cP)\not=(cp_1,cp_2)$. \par
\end{proof}\par
We note that D. I. Panyushev has recently determined the ``self-large'' (what we call ``stable'') nilpotent orbits in a quite general context of the Lie algebra
of a connected simple algebraic group over an algebraically closed field $K$ of characteristic zero
\cite[Theorem~2.1]{Pan}. When the Lie algebra $\mathfrak g$ of $G$ is $ {sl}(V)$
his result restricts to Theorem \ref{stablecor} above for such $K$ (ibid. Example 2.5 1.(a)).\par\smallskip
One of our goals is to show links between the study of nilpotent commuting matrices, and that of the punctual Hilbert scheme. Thus, we have taken pains to refer to some of the relevant literature that we were aware of. Nevertheless,
 the inclusion of proofs of Lemmas \ref{nilpBirredlem} and \ref{barnsthm}
 make the main results of the paper relatively self contained. We do use in an essential way the standard bases for ideals in $K\{x,y\}$ from \cite{Br,I1}, in the proof of
Theorem \ref{genericlem}.
\section{Pair of nilpotent matrices and the Hilbert scheme}\label{hilbsec}
We denote by $R=K\{ x,y\}$ the power series ring,
i.e. the completed local ring at
$(0,0)$ of the polynomial ring $K[x,y]$, over a field $K$. We assume throughout Section \ref{hilbsec} that $K$ is an algebraically
closed field, unless we state otherwise for a particular result. We denote by $M=(x,y)$ the
maximal ideal of $R$, and by $V$ the $n$-dimensional vector space over the field $K$ upon which
$M_n(K)$ acts. 
\par\noindent
\begin{definition} We denote by $\mathcal N(n,K)$ the set of nilpotent matrices in $M_n(K)$, with its natural structure as irreducible variety. We define $\mathcal H(n,K)$
\begin{equation*}
\mathcal H(n,K)=\{ (A,B)\mid A,B\in \mathcal{N}(n,K) \text { and } AB-BA=0\}.
\end{equation*}
Given an
element $(A,B)\in \mathcal H(n,K)$,
we denote by $\mathcal A_{A,B}\cong K[A,B]$ the Artinian quotient of $R$,  
\begin{align*}
\mathcal A&={\mathcal A}_{A,B}=R/I,\, I=I_{A,B}=\ker (\theta ),\\ \, \theta:& R\to k[A,B], \,\,\theta (x)=A, \theta (y)=B.
\end{align*}
We let
 $\mathcal U(n,K)\subset \mathcal H(n,K)$ be the open subset such that
$\dim_K(\mathcal A_{A,B})=n$.\par
\end{definition}
 The Hilbert scheme $A^{[n]}=\Hilb^n(\mathbb A^2)$ parametrizes length-$n$ subschemes of $\mathbb A^2$, and is a desingularization of 
the symmetric product $A^{(n)}=\Sym^n(\mathbb A^2)$. Given a point $s\in \mathbb A^2$, we denote by $H^{[n]}$ the fibre of $A^{[n]}$ over
the point $(ns)$ of $ A^{(n)}$: roughly speaking, the local punctual Hilbert scheme $H^{[n]}$ parametrizes the length-$n$ Artinian quotients of $R$.\footnote{Work of R. Skjelnes et al 
shows that this rough viewpoint is inaccurate, see \cite{LST}; the fibre definition is accurate.}
J. Brian\c{c}on and subsequently M. Granger of the Nice school, showed that
the scheme
$H^{[n]}$ is irreducible in characteristic zero
\cite{Br,Gra};
it was a slight extension to show $H^{[n]}$  is irreducible for $\cha K> n$
\cite{I1}, but further progress awaited a connection to $\mathcal H(n,K)$.\par
 V. Baranovsky, R. Basili, and A. Premet related this problem of irreducibility to that of the
irreducibility of ${\mathcal H}(n,K)$ \cite{Bar,Bas2,Prem}.
Following H. Nakajima and V.~Baranovsky, we set 
\begin{equation}\label{Ueqn}\mathfrak W\subset \mathcal H(n,k)\times V: \{ (B,A,v)\in \mathcal H(n,K)\times V\mid v \text { is a cyclic vector
 for }(B,A).\}
\end{equation}
That is, $(B,A,v)\in \mathfrak W$ if any $(B,A)$-invariant subspace of $V$ containing $v$ is all of $V$.
The group $Gl(V)$ acts on $H(n,K)\times V$ by conjugation of the matrices, and action on the vector. 
\begin{lemma}\label{freeactlem} (\cite[Theorem 1.9]{Nak}, \cite[Lemma 6]{Bar} The action of $Gl(V)$ on $\mathfrak W$ is free, and, taking
$x\to A, y \to B$, $x,y$ local parameters at $s\in \mathbb A^2$ we have a morphism,
\begin{equation}  
\pi: \mathfrak W \to H^{[n]},
\end{equation}
whose fibers are the $Gl(V)$ orbits in $\mathfrak W$.
\end{lemma}
\begin{theorem} \cite[Theorem 4]{Bar}\label{Barthm}\footnote{V. Baranovsky communicates in the MathSciNet review MR 1825165 of \cite{Bar} that a parenthetical remark
in the proof of Lemma 3, in (a) "i.e. $B_1$ has Jordan canonical form in this basis" is incorrect.}
The subset $\mathfrak W\subset \mathcal H(n,K)\times V$ is dense.
\end{theorem}\noindent
See also Lemma \ref{barnsthm} ff. As a consequence of Lemma \ref{freeactlem} and Theorem \ref{Barthm}, the irreducibility of $H(n,K)$ is equivalent to that of $H^{[n]}$. \par
 V. Baranovsky used this and Brian\c{c}on's Theorem to prove the irreducibility of $\mathcal H(n,K)$,
for $\cha K=0$ and $\cha K>n$. R.
Basili gave a direct ``elementary'' proof of the irreducibility of
$\mathcal H(n,K)$, that is valid also for $\cha K\ge n/2$. A. Premet later gave a Lie algebra
proof of the irreducibility of $\mathcal{H}(n,K)$ that is valid in all 
characteristics. The Basili and Premet results gave new (and different) proofs of the
irreducibility of $H^{[n]}$ when $K$ is algebraically closed, for $\cha K>n/2$ (R. Basili) or arbitrary characteristic (A. Premet).
 Note that the space of $\mathbb R$ (real) points of $\Hilb^n(R)$ has at least $\lfloor n/2\rfloor$ components
\cite[\S 5B]{I1}). These results showed that there is a strong connection between $\mathcal H(n,K)$ and $H^{[n]}$.
\subsection{Hilbert function strata:} 
Let $\mathcal A=R/I$ be an Artinian quotient of $R=K\{x,y\}$ of length $\dim_K(\mathcal A)=n\ge 1$,  and recall
that $M=(x,y)$ denotes the maximum ideal. The \emph {associated
graded algebra} $\mathcal A^*=Gr_M({\mathcal A})=\oplus_0^j {\mathcal A}_i$ of $\mathcal A$
satisfies (here $j=$ socle degree $\mathcal A:\,  \mathcal
A_j\not=0, {\mathcal A}_{j+1}=0$)
\begin{equation*}
{\mathcal A}_i=\langle M^i\cap I+M^{i+1}\rangle/M^{i+1}.
\end{equation*}
The
\emph{Hilbert function} $H(\mathcal A)$ is the sequence
\begin{equation*}
H(\mathcal A)=(h_0,\ldots ,h_j),\quad  h_i=\dim_K {\mathcal A}_i.
\end{equation*}
We denote by $n=|H|=\sum_i h_i$ the length of $H$, satisfying  $n=\dim_K (\mathcal A)$. 
\begin{example} Let $\mathcal A=R/I, I=(y^2+x^4,xy+
x^4)$. Then
\begin{equation}
{\mathcal A}^*=R/(y^2,xy,x^5),\text { and } H({\mathcal
A})=(1,2,1,1,1),
\end{equation} 
since
$x(\underline{y^2+x^4})-(y-x^3)(\underline{xy+x^4})=x^5+x^7\in
I\Rightarrow x^5\in I$.
\end{example}
 Let
$H$ be a fixed Hilbert function sequence of length $n$. We now study 
the connection between the Hilbert function strata
$Z_H=\Hilb^H(R)\subset H^{[n]}$, parametrizing all Artinian quotients of $R$ having
Hilbert function $H$, and the analogous subscheme of commuting pairs of matrices,
\begin{equation*}
{\mathcal H}^H(n,K)=\pi^{-1}(Z_H) =\{
\text {pairs }
(A,B) \mid
H(\mathcal A_{A,B})=H\}.
\end{equation*} Here $Z_H$ is locally closed in $H^{[n]}$ \cite[Proposition 1.6]{I1}, and likewise so is ${\mathcal H}^H(n,K)$ in $\mathcal U(n,K)$.  We have the projection
\begin{equation*}
 \tau : Z_H \to G_H ,  \mathcal A\to \mathcal A^*
\end{equation*}
to the irreducible projective variety
$G_H$ parametrizing graded quotients of $R$ having Hilbert function $H$. Each of $Z_H,G_H$ have covers by opens in affine spaces
 of known dimension \cite{Br,I1}; also $\tau $ makes $Z_H$ a locally trivial bundle over $G_H$ with fibres opens in an affine space, and having a 
global section \cite[Theorems 3.13, 3.14]{I1}, but $Z_H$ is not in general a vector bundle over $G_H$ \cite{I1.5}. When $\cha K = 0$ or $\cha K= p>n$
the fiber is an affine space and the covers are by affine spaces \cite[Theorems 2.9, 2.11]{I1}.
 The Nice school studied specializations of
$Z_H$, see work of M.~Granger \cite{Gra} and J.~Yam\'{e}ogo \cite{Yama, Yamb}, but the problem of
understanding the intersection $\overline{Z_H}\cap Z_{H'}$ is in general difficult and
quite unsolved (see \cite{Gu,NV} for some recent progress). Let $Z_{\nu ,n}$ parametrize order $\nu$ colength $n$
ideals $I$ in $R=K\{ x,y\}$:
that is 
\begin{equation*}
Z_{\nu ,n}=\{I\mid M^\nu \supset I, 
M^{\nu+1}\nsupseteq I,{\text { and }} \dim_K R/I=n\}.
\end{equation*}
 J.~Brian\c{c}on's irreducibility result can be stated, denoting by $\overline X$ the Zariski closure of $X$, 
\begin{equation*}
H^{[n]}=\overline{Z_{1,n}}.
\end{equation*}
M. Granger showed, more generally
\begin{theorem}\label{Grangerthm}\cite{Gra} For $\nu\ge 1$ we have
\begin{equation}\label{Grangereq}
\overline{Z_{\nu ,n}}\supset Z_{{\nu+1},n} .
\end{equation}
\end{theorem}
We let $\mathcal U_{\nu ,n} =\pi^{-1}(Z_{\nu,n} )$.
\begin{corollary}\label{closureUnucor} Fix $n$.  Then for $\nu\ge 1$ we have
\begin{equation}
\overline{\mathcal U_{\nu ,n}}\supset \mathcal U_{\nu+1,n}.
\end{equation}
\end{corollary}
\begin{proof} This is an immediate consequence of Granger's theorem and Lemma \ref{freeactlem}.
\end{proof}
\par\noindent Recall that when an Artinian algebra $A$ has embedding dimension at most two ($h_1\le 2$) its  Hilbert function $H(A)$ 
satisfies, (see \cite{Mac1, Br,I1})
\begin{equation}\label{cod2eq}
H=(1,2,\ldots
,\nu,h_\nu,\ldots ,h_j), \nu\ge h_\nu\ge
\ldots
\ge h_j>0,
\end{equation}
Here, writing $A=R/I, R=K\{ x,y\}$, we have that $\nu$ is the \emph{order} $\nu (I)$ of the ideal $I$, namely the smallest initial degree of any element of $I$. (When $\nu (I)=1$,  $H=(1,1,\ldots
,1)$: we regard this as also a sequence satisfying \eqref{cod2eq}.)\par
 Henceforth, by \emph{Hilbert function} we will mean one of codimension at most two, so a 
sequence satisfying \eqref{cod2eq}. The \emph{length} of a Hilbert function is $n=\sum_0^j h_i$. The \emph{socle degree} of $H$ is the integer $j$ from \eqref{cod2eq}, and it
is also the socle degree -- maximum nonzero power of the maximal ideal -- of any Artinian algebra of Hilbert function $H$.
\begin{definition}\label{diagdef} Recall that we arrange the Ferrer's graph (Young diagram) of the partition $P=(p_1\ge \cdots \ge p_t)$ with 
the largest row of length $p_1$ at the top. The \emph{diagonal lengths} $H_P$ of a
partition $P$ are the lengths of the lower left to upper right diagonals of the
Ferrer's graph of $P$. The \emph{dual partition} $\hat{P}$ to $P$ is obtained by switching rows and columns in the Ferrer's graph: 
\begin{equation} \hat{P}_i=\#\{p_k\in P\mid p_k\ge i\} .\end{equation}  \par Given a Hilbert function $H$ as in \eqref{cod2eq}, we denote by
$P(H)$ the unique partition having diagonal lengths $H$ and $\nu$ strictly decreasing parts.  It satisfies
$P(H)=(p_1,\ldots )$ with $p_i$ the length of the
$i$-th row of the bar graph of $H$. In other words, were the sequence $H$ rearranged in descending order, then
$P(H)$ would be the dual partition to $H$.
\end{definition}
\begin{example}
  For $H=(1,2,3,2,1),\,
P(H)=(5,3,1)$. The partitions $P=(4,2,1,1,1)$ and $P'=(3,3,3)$ also have diagonal lengths $H=(1,2,3,2,1)$, but are incomparable in the partial order \eqref{poparteq}.
We show below that $P(H)$ is maximum among the partitions of diagonal lengths $H$.
\end{example}
\begin{remark} Fixing a Hilbert function $H$, the elements of the set $\mathcal P(H)$ of partitions having diagonal lengths $H$ with a certain grading (by the number of difference-one hooks), correspond bijectively to the
cells in a cellular decomposition of the projective variety $G_H$, graded by dimension (see \cite{IY}). The Hilbert function $H$ determines a certain product $B(H)$ of rectangular partitions; and the elements of
 $\mathcal P(H)$ correspond bijectively to sequences of subpartitons of $B(H)$, in what is termed a ``hook code'' in \cite[Section 3D]{IY}. Thus, $\mathcal P(H)$ is enumerated by a certain product of 
binomial coefficients \cite[Theorem 3.30]{IY}.
\end{remark}\par\smallskip
 There is a natural partial order on the set $\mathcal {H}(n)$ of
Hilbert functions of codimension at most two, having length
$n$ (see \eqref{cod2eq}), given by 
\begin{equation}\label{poshilbeq}
H\le H' \Leftrightarrow \forall u, 0\le u<n, \sum_{k\le u} H_k\le \sum_{k\le u} H'_k.
\end{equation}
For example, $(1,1,1,1,1)< (1,2,1,1)<(1,2,2)$. \par
 The maximality of $P(H)$ in the next Lemma \ref{PHlem}\eqref{PHlemB} follows from the irreducibility of $G_H$ and considering the cells corresponding to each partition $P$ of diagonal lengths $H$ in $G_H$
(see \cite[Theorem 3.12ff]{IY}).  We include a simple direct proof of \eqref{PHlemB}.
\begin{lemma}\label{PHlem}\begin{enumerate}[i.]
\item\label{PHlemA} The assignment $H\to P(H)$ determines an order-reversing bijection between the partially ordered set (POS) of Hilbert functions of length $n$ (see
\eqref{poshilbeq}), and the POS of partitions of $n$ having
decreasing parts (see \eqref{poparteq}).
\item\label{PHlemB} Let $P$ have diagonal lengths $H$. Then $P(H)\ge P$ in the partial order \eqref{poparteq}.
\end{enumerate}
\end{lemma}
\begin{proof}  We first show  \eqref{PHlemA}. It is well known that the correspondence taking $P$ to $\hat{P}$
is an order reversing involution on the POS of partitions of $n$ \cite[Lemma 6.3.1]{CoM}. It takes a partition $P=(p_1,p_2, \ldots ,p_v)$ having $\nu$ decreasing parts, to a partition $Q=\hat{P}$
 having no gaps among the integers $(1,2,\ldots ,\nu )$. Given $P$ we let $H_P=(1,2,\ldots v, h_v,\ldots ,h_j)$, be the sequence obtained by rearranging  $\hat{P}$, so that
$H_P$ begins $(1,2 \dots ,\nu)$, and ends with the rest of the parts of $\hat{P}$ in non-increasing order. 
Then $H_P$ satisfies \eqref{cod2eq}, and $P(H_P)=P$. Evidently, $P\to H_P$ is a bijection as stated in \eqref{PHlemA}. \par It remains to show that if two partitions $Q=\hat{P}$, and $Q\, '=\hat{P\,'}$ with maximum parts $\nu,\nu'$ respectively, and having no gaps among
$(1,2,\ldots ,\nu )$ and $(1,2,\ldots ,\nu\, ')$, respectively, satisfy $Q\le Q\,'$ then the rearranged sequences $H$, $H'$ satisfy $H\le H'$ 
in the order \eqref{poshilbeq}, and vice-versa. It is well known that the partial order between two partitions of $n$ is preserved
by the operation of either removing (or, respectively, adding) a common part $a$ to each, forming partitions of $n-a$ (or, respectively $n+a$).  Removing in this way the
parts $(1,2, \ldots ,\nu)$, and placing those parts first leaves remainder partitions  $\alpha(Q)\le \alpha (Q\,')$. Now the first sequence is $H=(1,2,\ldots, \nu , \alpha (Q))$ and we 
have $H\le (1,2,\ldots \nu, \alpha (Q'))$ in the partial order obtained by formally extending that of \eqref{poshilbeq} to arbitrary sequences.
Finally, rearranging the parts $(\nu +1,\ldots ,\nu\,')$ (if any) of $\alpha(Q')$ first, (so just after $\nu$), puts the second sequence in the form $H'$; we have $H\le H'$  since each of $\nu+1,\ldots \nu'$ are larger
 than every part of $\alpha (Q)$. This argument reverses, showing that the mapping $H\to P(H)$ inverts the partial order. This completes the proof of \eqref{PHlemA}.

\par
To show \eqref{PHlemB}, let $P: p_1\ge ... \ge p_s$ have diagonal lengths $H$ and consider the partition $P(u)=(p_1,\ldots ,p_u)$ comprised of the first $u$ rows of $P$. Rearranging the rows of the Ferrer's graph of $P(u)$ in staggered fashion,
by advancing the $v$-th row from the top (longest) by 
$v-1$, and forming an adjusted Ferrer's graph $\AFG (P(u))$ we see that the sequence $H(u)$ given by the diagonal lengths of $P(u)$ is given by 
\begin{equation}
H(u)_i= \text { the length of the $i$-th column of } \AFG (P(u)).
\end{equation}
The partition $P(H(u))$ is obtained by pushing all squares in $\AFG (P(u))$ upward, so that in each column there are no gaps. Thus, $P(H(u))$ partitions the same number $\mid P(u)\mid$ as
$P(u)$. Since $H_i\ge H(u)_i$ for each $i$, the Ferrer's graph of the partition $P(H)$ includes that of $P(H(u))$ (strictly if $p_{u+1}\not=0$). Thus for each $u$
\begin{equation}
\sum_{k\le u} p_k = \sum_{k\le u} P(H(u))_k=\mid P(u)\mid \,\le \sum_{k\le u} P(H)_k,
\end{equation}
showing that $P\le P(H)$ in the partial order \eqref{poparteq}.
\end{proof} \par
\begin{example}\label{posexample} The partitions $P=(6,4,3), P'=(6,4,2,1)$ with decreasing parts satisfy $P\ge P'$, so their duals $\hat{P}=(3,3,3,2,1,1), \hat P'=(4,3,2,2,1,1)$
satisfy $Q=\hat P\le Q\,'=\hat P\,'$. Since $\nu=3$ this implies that $\alpha (Q)=(3,3,1)\le \alpha (Q\,')=(4,2,1)$ in the POS of equation \eqref{poparteq},
 implying $H_P=(1,2,3,3,3,1)\le \, H_{P'}=(1,2,3,4,2,1)$ in the POS of \eqref{poshilbeq}.
\end{example}
Let $I$ be an ideal of colength $n$ in $R=K\{x,y\}$ and let $H=H({\mathcal A}),
{\mathcal A}=R/I$. Recall $\nu=$ order of
$I $; so $  M^\nu\supset I, M^{\nu +1}\nsupseteq I, $ where $M=(x,y).$ Consider the deg lex partial order, 
\begin{equation*}
1<y<x<y^2<yx<x^2\cdots
\end{equation*}
 and 
denote by $E=E(I)$ the monomial initial ideal of $I$ in this order.  The monomial cobasis $E(I)^c =\mathbf{N}^2-E(I)$ may be
seen as the Ferrer's graph of a partition $P=P(E)$ of diagonal lengths
$H$. Conversely, given a partition $P=(k_0,\ldots ,k_{\nu-1})$ with $\nu$ nonzero parts (the notation is
from the standard bases introduced just below in Definition \ref{stdbasisdef}), we define the monomial ideal $E_P$
\begin{equation}\label{mpeq}
E_P=(x^{k_0}, yx^{k_1}, y^2x^{k_2}, \ldots, y^{\nu-1}x^{k_{\nu-1}}, y^\nu), 
\end{equation}
whose cobasis $E_P^c$ is the complementary set, of monomials $E_P^c=\mathbf{N}^2-E_P$ (where the pair of non-negative integers $(a,b)\in \mathbf{N}^2$ denotes $x^ay^b$).
\begin{definition}\label{stdbasisdef} The ideal $I\subset R=K\{ x,y\}$ has 
\emph{standard basis} $(f_\nu,\ldots ,f_0)$ in the direction $x$ if $I$ has a (not necessarily minimal)
generating set $(f_0,\ldots ,f_\nu)$ of the following form.  
\begin{align}\label{stdbasiseq}
&(f_\nu=g_\nu, f_{\nu -1}=x^{k_{\nu -1}}g_{\nu -1},\ldots ,f_0=x^{k_0}g_0),
\text { where }\\ 
 g_i&=y^i+
h_i,\,\, h_i\in M^i\cap k[x]\langle y^{i-1},\ldots, y, 1\rangle\notag
\end{align}
and $k_0\ge k_1\ge \ldots \ge k_{\nu -1}$ \cite[Definition 3.9ff]{IY}. We term the basis \emph{normal} if $k_0 > k_1> \ldots > k_{\nu -1}$ \cite{Br,I1}.
We will sometimes refer to these as ``standard generators', or ``normal generators'', respectively.
\end{definition}
Then we have $E=E(I)$ is the monomial ideal of \eqref{mpeq} and $E^c$ is the set of monomials
\begin{equation}\label{cobasiseq}
E^c=\langle 1,x,\ldots x^{k_0-1}; y,yx,\ldots yx^{k_1-1}; \ldots ;y^{\nu-1}, \ldots ,y^{\nu-1}x^{k_{\nu-1}-1}\rangle .
\end{equation}
The existence of a normal basis in the direction $x$ does not depend on 
the choice of $y\in R_1$, such that $\langle y,x\rangle=R_1$.
Note also that for a normal basis the decreasing sequence $P=(k_0,k_1,\ldots ,k_{\nu-1})$ satisfies $P=P(H)$, where
$H=H(R/I)$ is the Hilbert function of ${\mathcal A}=R/I$. \par
The following result is standard, see for example \cite[Lemma 1.4]{I1}. We denote by $\langle E^c\rangle$ the $K$-vector space spanned by $E^c$.
\begin{lemma}\label{standbasiseqlem} The condition \eqref{stdbasiseq}
 is equivalent to 
\begin{equation}\label{complementeq}
\forall i\ge 0, \langle E^c\rangle\cap M^i\oplus I\cap M^i=M^i, \text { an internal direct sum}.
\end{equation}
\end{lemma}
This notion of standard basis is stronger than just ``$E^c$ is a complementary
basis to $I$ in $R\,$'', used in \cite{BH,NS}.\par The following Lemma is well known, for
example \cite[Lemma 3]{Bar} shows that for a generic $A$ in $\mathcal N_B$ the pair $(A,B)$ has a cyclic vector, and by
\cite{NS} this implies $\dim_K K[A,B]=n$. We thank A.~Sethuranam and T. Ko\v{s}ir for discussions of these topics that led to our
proof below.
\begin{lemma}\label{barnsthm}\begin{enumerate}[i.]
\item\label{barnsthmi} Let $B$ be an $n\times n$ nilpotent Jordan
matrix of partition $P$ and let
$A$ be generic in
$\mathcal N_B$. Let $K$ be an infinite field. Then
\begin{equation*}
\dim_K K[A,B]=n.
\end{equation*}
\item\label{barnsthmii}\cite[Lemma 3]{Bar}. Let $B $ be nilpotent, and $C\in \mathcal N_B$ and assume $K$ is algebraically closed. Then there exists $A\in \mathcal N_B$ such that the pencil $A+tC\subset \mathcal N_B$, and the pair $(A,B)$ has
a cyclic vector. 
\end{enumerate}
\end{lemma}
\begin{proof} For \eqref{barnsthmi} consider the monomial ideal $E_P$; then the matrix of
$B=\times x$ acting on the basis $E_P^c$ of \eqref{cobasiseq} is the Jordan matrix of partition $P$; the matrix of $A=\times y$ has the conjugate Jordan
 partition $\hat{P}$, and
$
\dim_K K[A,B]=n$. Now $\dim_K K[A,B]$ is upper semicontinuous on $A\in \mathcal N_B$, an irreducible variety (Lemma \ref{nilpBirredlem}), and the dimension of the algebra generated by any two commuting
$n\times n$ matrices is less or equal $n$ (\cite{Ge}, see also \cite{Gur,GurSe}).\par
For \eqref{barnsthmii}, since $K$ is closed, wolog we may assume that $B$ is in Jordan form. By \cite[Lemma 2.3]{Bas2} there is an element $g\in \mathcal {C_B}^*$ such that $gCg^{-1}\in S_B$, where $S_B$ is a maximal nilpotent
subalgebra of $\mathcal N_B$. Let $A'$ be a general enough element of $\mathcal S_B$, and take $A=g^{-1}A'g$. Then the pencil $A+tC\subset g^{-1}\mathcal S_Bg\subset \mathcal N_B$; and \eqref{barnsthmi} implies $\dim_K K[A,B]=n$.
\end{proof}\par 
V. Baranovsky uses \eqref{barnsthmii} to show that the subset $\mathfrak W$ of $\mathcal H(n,K)\times V$ consisting of triples $(A,B,v)$ for which $v$ is a cyclic vector for the pair $(A,B)$ is dense (Theorem \ref{Barthm}).
\subsection{Pencil of matrices and Jordan form}\label{pencilsec} 
\par  We first give an example illustrating the connection between Hilbert
function strata $Z_H$  of Artinian algebras and those of commuting nilpotent matrices.
Here are some features.
Assume $K[A,B]\in {\mathcal H}^H(n,K)$. Then
\begin{enumerate}[i.] 
\item The ideals that occur in writing $K[A,B]\cong R/I$ are in general non-graded.
\item The partition $P$ need not have diagonal lengths $H=H(K[A,B])$, and $P(H)\ge P$ in \eqref{poparteq}.
\item The partition $P_\lambda$ arising from the action of $A+\lambda B$
satisfies $P_\lambda = P(H)$ for a generic $\lambda$, all but a finite number (Theorem \ref{genericlem}).
\item The closure of the orbit of $P$ includes a partition of diagonal lengths $P(H)$ (Theorem \ref{hilbfuncthm}).
\end{enumerate}
\begin{example}[Pencil and specialization]\label{specializeex}
Take for
$B$ the Jordan matrix of partition $(3,1,1)$. 
It is easy to see that for $P=(3,1,1)$ we have $Q(P)=(4,1)$, Also by \cite[Lemma 2.3]{Bas2}, up to conjugation by an element of the centralizer $\mathcal C_B$, any element
$A\in
\mathcal N_B$ satisfies
\large
\begin{equation*}
B=\left( 
\begin{array}{ccc|cc}
0&1&0&0&0\\
0&0&1&0&0\\
0&0&0&0&0\\
\hline
0&0&0&0&0\\
0&0&0&0&0\\
\end{array}
\right),
\quad A=\left( 
\begin{array}{ccc|cc}
0&a&b&f&g\\
0&0&a&0&0\\
0&0&0&0&0\\
\hline
0&0&e&0&c\\
0&0&d&0&0\\
\end{array}
\right).
\end{equation*}
\normalsize
We send $x\to A, y \to B$, and let the ideal $I=\Ker (R\to K[A,B])$. We now assume that $acdf\not=0$ so that $A$ be general enough to have Jordan block partition $Q(P)$. Let $\beta=1/(cdf)$,
and let
$$ g_2=y^2-\beta x^3, g_1=y-a\beta x^2, g_0=1.
$$
Then $I$ has a normal basis in the $x$ direction (Definition  \ref{stdbasisdef}): we have 
$$
\mathcal A=\mathcal A_{A,B}=K[A,B]\cong R/I, I  = (g_2,xg_1,x^4g_0).
$$
with $k_0=4, k_1=1$ in \eqref{stdbasiseq}, and the Hilbert function $H(\mathcal A)=(1,2,1,1)$. The multiplication action $A=m_x$ of $x$
on the classes
$\langle 1,x,x^2,x^3;g_1 \rangle$ in $\mathcal A$ has Jordan blocks given by the
partition
$(4,1)$ having
 diagonal lengths $H(\mathcal A)$.\par 
We have in the $y$-direction $I=(x^3-\beta^{-1} y^2, xy-ay^2,y^3)$: the non-homogeneous generator $x^3-\beta^{-1} y^2$ with lead term $x^3$ prevents
$I$ from having a standard basis in the direction $y$. The
action of
$B=m_y$ on the classes of $\langle 1,y,\beta x^3;x-a y, y^2 \rangle$ in
$\mathcal A$ verifies that
$P_B =(3,1,1)$ of diagonal lengths $(1,2,2)$, which is \emph{not}
$H(\mathcal A)$. 
\par
 Now
consider the associated graded algebra $\mathcal A^*=R/I^*$: here
$I^*=(y^2,xy,x^4)$.  The standard generators in the $y$ direction (switch $y,x$ in the Definition \ref{stdbasisdef}) are
$(x^4,x^3y, x^2y,xy,y^2)$.
The action of $m_y$ on the $K$-basis $\langle 
1,y; x,x^2,x^3\rangle$ of $\mathcal A^*$ has Jordan partition $P'=(2,1,1,1)$ of diagonal 
lengths
$H({\mathcal A})=(1,2,1,1)$ (Lemma \ref{gradpartlem}). (In the $x$ direction
$I^*$ has normal generators $(y^2,xy,x^4)$ of partition (4,1), the same partition as for $I$.)
  Also, holding $a$ constant, we have
\begin{equation*}
I^*=\lim _{\beta\to 0}
I,
\end{equation*}
so $P'=(2,1,1,1)$ is in the closure of the orbit of $B$ (Theorem \ref{hilbfuncthm}\eqref{hilbfuncthmii}.)\par
Here $\dim G_H=1$: a graded ideal of Hilbert function $H$ must satisfy 
\begin{equation*}
\exists L\in
R_1\mid I=(xL,yL,M^4),
\end{equation*}
 so
$G_H\cong \mathbb P^1$, and $I\in G_H$ is determined by the choice of the linear form $L$,
here $L=y$. The fibre of $Z_H$ over a point of $G_H$ is determined here by the choice of
$a, \beta$, so has dimension two.
\end{example}
\begin{theorem}\label{genericlem} Assume $A,B$ are commuting $n\times n$ nilpotent matrices
with
$B$ in Jordan form and suppose $\dim_K K[A,B]=n$. Let $H= H(K(A,B))$ be the Hilbert function. Let 
$K $ be an algebraically closed field of characteristic zero, or of characteristic $p>j$ the socle degree of 
$H$. Then for a generic $\lambda\in \mathbb P^1$, the Jordan block sizes of the
action of 
$A+\lambda B$ both on $K[A,B]\cong R/I$ and on the associated 
graded algebra $Gr_M K[A,B]\cong Gr_M(R/I)$, are given by the parts of $P(H)$. We have $P(H)\ge P$ in the POS of \eqref{poparteq}.
\end{theorem}
\begin{proof} By \cite{Br} in the case $\cha K=0$ or \cite{I1} when $\cha K=p>j$, there is an open dense set of $\lambda \in \mathbb A^1$, such that 
the ideal $I$ has normal
basis in the direction $x'=x+\lambda y$.  Replacing $x$ in \eqref{stdbasiseq} by $x'$, so considering the standard basis $f_0,\ldots ,f_{\nu -1}$ there, and considering the action of
$m_{x'}=\times x'$ on the cyclic $K[x']$ subspaces of $R/I$ generated by $1,g_1,\dots g_{\nu
-1}$, we see that the Jordan partition of 
$m_{x'}$ is just $P(m_{x'})=(k_0,\ldots ,k_{\nu -1})$. This is $P(H)$ since the basis is normal.\par
The standard basis for the associated graded ideal is given by the initial ideal $In I$,
satisfying
\begin{equation*}
In I=(In(f_\nu),\ldots ,In(f_1),f_0),
\end{equation*}
where here  $In f$ denotes the lowest degree graded summend of $f$. 
 So the Jordan partition for the action of
$m_x$ on $R/I^*$ is also $P(H)$.
\end{proof}\par
We thank G. McNinch for comments and a discussion that led to the following corollary. The corollary implies the special case of his result [McN, Theorem 26]
where $ K[A,B] $ is assumed cyclic, and also $K$ is algebraically closed of suitable characteristic.
\begin{corollary}\label{GMcor} Assume that $A,B$ and the field $K$ satisfy the hypotheses of Theorem \ref{genericlem}. Then for generic $t$, $A$ and $B$ are in the Jacobson radical of $\mathcal C_t$, the commutator algebra of $A+tB$.
\end{corollary}
\begin{proof} The Jordan partition $P_t$ given by the blocks of $A+tB$ for $t$ generic is strictly
decreasing, as it has the form $P(H)$. That the partition $P_t$ has distinct parts is equivalent to the semisimple quotient $\mathcal C_t/\mathfrak{J}_t$ of the commutator algebra $\mathcal C_t\subset \mathrm{End}V$ of $A+tB$
satisfying $\mathcal C_t/{\mathfrak J}_T$ being an \'{e}tale algebra -- a product
of fields $K$, one copy for each distinct part of $P_t$ \cite[Lemma 2.3]{Bas2}. Thus $A$ and $B$, being nilpotent, are in the Jacobson radical
of $\mathcal C_t$. 
\end{proof}\par
 The following example communicated to us by G. McNinch shows that the restriction on $\cha K$ in Theorem~\ref{genericlem} is sharp.
\begin{example}\label{McNex}
Let $d$ be a positive integer. Let $V_1$ be a $d$-dimensional $K$-vector space, let $V_2$ be a $ 2$-dimensional $K$-vector space, and let $V$ be the tensor product $V= V_1 \otimes V_2$.
Let $A = J_d \otimes I_2$, where $J_d$ is a Jordan block of size $d$ in the $V_1$ factor and $I_2$ the identity. So the partition of $A$ is $(d,d)$. And let $B = I_d \otimes J_2$.
Then $A$ and $B$ commute. The algebra $K[A,B]$ is isomorphic to $K[x,y]/(x^d,y^2)$, and has vector space dimension $n=2d$. Its Hilbert function is
$H = (1,2,2,...,2,2,1)$ of socle degree $d$,  so $ P(H) = (d+1,d-1)$. (This is the answer one expects from the rule for computing tensor products
of representations of the Lie algebra $sl(2))$). If the integer $d$ is invertible in K, the partition of $A + tB $ is indeed $(d+1,d-1)$
for all $t$ not zero. But in characteristic $p$ dividing $d$, the
Jordan block partition of $A + tB$ is $(d,d)$ for all $t$ \cite[Example 22]{McN}. 
\end{example}
We isolate a result that can be concluded simply from \cite[Definition 3.9,Theorem 3.12]{IY} or from Gr{\"o}bner basis theory. We use the notation
from Definition \ref{stdbasisdef}.
\begin{lemma}\label{gradpartlem} Let $\mathcal A$ be a graded Artinian algebra quotient $\mathcal A=R/I$ of $R$, so $\mathcal A=\oplus_0^j \mathcal A_i$. Then we have
\begin{enumerate}[i.]
\item \label{gradpartlemi} Let $x\in \mathcal{ A}_1$. Then $I$ has a
standard basis in the direction $x$.
\item\label{gradpartlemii}  The partition $P'$ given
by the Jordan blocks of the action of $x$ on $\mathcal A$ satisfies $P'=(k_0,\ldots k_{\nu -1})$ from \eqref{stdbasiseq}, and has diagonal lengths $H=H(\mathcal A)$.
\end{enumerate}
\end{lemma}
\begin{proof} The initial monomial ideal $E(I)$ in the $x$-direction certainly has a basis as in \eqref{stdbasiseq}, for some sequence of integers $k_0, \ldots ,k_{\nu -1}$: to show a standard
basis we must show that the sequence is non-increasing. However, if $k_u>k_{u-1}$ then multiples of $yf_{u-1}$ could be used to eliminate $y^ux^{k_u}$ from the initial ideal $E(I)$.\par
Then \eqref{gradpartlemii} follows from \eqref{stdbasiseq}, since $\mathcal A$ is the internal direct sum of the $k[X]$ modules generated by  $1,g_1,\ldots ,g_{\nu -1}\in\mathcal A$.
  
\end{proof}\par

Recall that $\mathcal U_B$ is the open dense subset of $\mathcal N_B$ for which
$\dim_K K[A,B]=n$. Now using the connection between
$Z_H$ and $\mathcal H^H(n,K)$ we have
\begin{theorem}\label{hilbfuncthm} Let $B$ be nilpotent with Jordan partition $P$, let $A\in \mathcal U_B$, and let $H=H(K[A,B])$. Suppose that $K$ is as in Theorem~\ref{genericlem}. Then 
\begin{enumerate}[i.]
\item\label{hilbfuncthmi} For generic $\lambda\in \mathbb
P^1$ the Jordan block sizes of the action of $A+\lambda B$ on $K[A,B]$ are given by the parts of $P(H)$.
\item\label{hilbfuncthmii} The closure of the $\mathrm {Gl}_n$ orbit of $B$ contains a
nilpotent matrix having partition $P'$ whose diagonal lengths are given by $H$.
\end{enumerate}
\end{theorem}
\begin{proof} We may assume that $B$ is in Jordan form. It follows from the assumptions and Theorem \ref{genericlem} that
$C_\lambda=A+\lambda B$ for
$\lambda$ generic satisfies, $P(C_\lambda)=P(H)$. Since the algebra $\mathcal A=\mathcal
A_{A,B}=K[A,B]$ is a deformation of the associated graded algebra $\mathcal A^*$, the 
multiplication $m_y$ on $\mathcal A$ is a deformation of the action $m_y$ on $\mathcal A^*$,
so the orbit $P'$ of the latter is in the closure of the orbit of $P$. By Lemma~\ref{gradpartlem}~\eqref{gradpartlemii} $P'$ has
diagonal lengths $H$.
\end{proof}
\begin{theorem}\label{decreasebthm}  Let $B$ be nilpotent of partition $P$, and denote by $Q(P)$
the partition giving the Jordan block decomposition for a generic element $A\in
\mathcal N_B$. Suppose that $K$ is algebraically closed and that $\cha K=0$ or $\cha K >n$. Then
$Q(P)$ has decreasing parts and is the greatest $P(H)$ that occurs for Hilbert functions
of length $n$ algebras $\mathcal A=K[A,B]$, with $A\in \mathcal N_B$:
\begin{equation*}
Q(P)=\sup \{ P(H) \mid \exists A\in \mathcal U_B, H=H(K[A,B])\}.
\end{equation*}
\end{theorem}
\begin{proof} By the irreducibility of
$\mathcal N_B$ (Lemma \ref{nilpBirredlem}), there is an orbit $Q(P)$ whose closure contains each other orbit occuring in $\mathcal U_B$. By Theorem \ref{hilbfuncthm}
$Q(P)$ has the form $P(H)$ for some $H$. Since the closure of its orbit contains each other orbit, this $P(H)=Q(P)$ is greater than every other $P(H')$
 for a sequence $H'$ among $\{ H(K(A,B)), A\in \mathcal U_B$\}. 
\end{proof}\par\bigskip
 We recall the natural order \eqref{poshilbeq} on the set $\mathcal {H}(n)$ of
Hilbert functions of length $m$.
 \par\noindent The openness on $\Hilb^n(R)$
of the condition 
\begin{equation*}
\dim_K I\cap M^{u+1}>s
\end{equation*}
  shows that 
\begin{equation}\label{hilbspecialeq}
\overline{Z_H}\cap Z_{H'}\not= \emptyset\, \Rightarrow H\le H'.
\end{equation}
We denote by $\mathfrak W_B$ the fibre over projection on the first factor of $\mathfrak W$ from \eqref{Ueqn} : thus $\mathfrak W_B$ is isomorphic to pairs $(A,v)$ with $v$ a cyclic vector
for $(B,A)$; and it is acted on by the units $\mathfrak G={\mathcal C_B}^*$ of the commutator $\mathcal C_B$ of $B$ by: $g\in \mathfrak G \Rightarrow
g(A,v)=(gAg^{-1}, g(v))$. Recall that $B=J_P$, and that $i(Q(P))$ is its largest part.
\begin{lemma}\label{comparelem} We have the following:
\begin{enumerate}[i.]
\item\label{comparelemi} Let $(A,v),(A',v')\in \mathfrak W_B$ satisfy, the closure of the $\mathfrak G$ orbit of $(A,v)$, contains that of $(A',v)$.  Then $H(K[A,B])\le H(K[A',B])$.
\item\label{comparelemii} Let $(A,v)\in \mathfrak W_B$ satisfy $P_A=Q(P)$, and let $K$ satisfy, $\cha K=0$ or $\cha K> n$. Then $H(K[A,B])=H_{Q(P)}$, the diagonal lengths of $Q(P)$.
\end{enumerate}
\end{lemma}
\begin{proof} The claim \eqref{comparelemi} follows from $\pi: \mathfrak W_B\to H^{[n]}$ being a morphism, and \eqref{hilbspecialeq}. Concerning \eqref{comparelemii}, let $A$ have partition $Q(P)$,
and let $H=H(K[A,B])$. By Theorem \ref{hilbfuncthm}, for generic $\lambda\in \mathbb
P^1$ the Jordan block sizes of the action of $A+\lambda B$ on $K[A,B]$ are given by the parts of $P(H)$. Since $m_{A+\lambda B}$ specializes to $m_A$, we have $Q(P)\le P(H)$. But $Q(P)$ is the
partition of the generic element $A\in \mathcal N_B$, which is irreducible, so $Q(P)\ge P(H)$, implying equality. By Lemma \ref{PHlem} \eqref{PHlemA} $H=H_{Q(P)}$.
\end{proof}
\par
Note that an analogous result to Lemma \ref{comparelem}\eqref{comparelemii} would hold for any irreducible subset $\mathcal N$ of $\mathcal N_B$, satisfying $A \in \mathcal N \Rightarrow $ the pencil $A+tB\subset  N, t\in K$.
\begin{theorem}\label{qphilb}
Let $B$ be Jordan of partition $P$ and let $\cha K=0$ or $\cha K> n$. Then
\begin{equation*}
Q(P)=P(H_{min}(P)),\text { where } H_{min}(P)=\min \{ H\mid \exists A\in \mathcal U_B \mid H(K[A,B])=H\}.
\end{equation*}
\end{theorem}
\begin{proof}  By Lemma \ref{PHlem} \eqref{PHlemA} the bijection $H\to P_H$ from Hilbert functions to
partitions with decreasing parts, is order-reversing. The assertion thus follows from Theorem \ref{decreasebthm}. 
\end{proof}\par\noindent
Note that Lemma \ref{comparelem}\eqref{comparelemi} may be used in place of Lemma \ref{PHlem} in the above proof.
\begin{remark} P. Oblak has shown a formula for the index $i(Q(P))$, which is the largest part of $Q(P)$. This was proven in \cite{Oblak} for $\cha K=0$,
but can be shown valid in all characteristics \cite{BI1}. For a Hilbert function $H$, the index $i(P(H))$ is by definition one greater 
than the socle degree $j$ of $H$. This suggests that Theorems \ref{decreasebthm} and \ref{qphilb} might hold for $\cha K\ge i(Q(P))$.
\end{remark}
  T. Ko{\v{s}}ir and P. Oblak have recently resolved the question we asked in \cite[p.3]{I4} whether $Q(P)$ is stable  (Theorem \ref{Gorensteinthm}). We give a short summary in order to comment on the relation of their
result to the Hilbert scheme. An insight they had was that the question about stability is closely related to the case $e=2$ of the following classical result about height two ideals.
\begin{lemma}\label{genboundlem} Let $K$ be an infinite field and $\mathcal A=R/I, R=K\{ x,y\}$ be an Artinian quotient.  \begin{enumerate}[i.]
\item\label{Gorci} Then $A$ satisfies $\dim_K (0:m)=e-1$ if and only if the ideal $I$ has $e$ generators in a minimal generating set.
\item\label{genii} Let $I$ have $e$ generators in a minimal generating set. Then the Hilbert function $H(A)$ satisfies,
\begin{equation*} i\ge \nu (I) \Rightarrow h_{i-1}-h_i\le e-1.
\end{equation*}
 In
particular, if $I$ is a complete intersection (e=2) then $h_{i-1}-h_i\le 1$.
\end{enumerate}
\end{lemma}
\begin{proof}[Comment on proof] The result \eqref{Gorci} is shown by F.H.S. Macaulay in \cite{Mac1} following earlier articles \cite{Mac0,Sc},
that were incomplete. The case $e=2$ is that an Artinian ring $A$ is Gorenstein of codimension two if and only if $A$ is a complete intersection (CI). The usual proof given now uses the Hilbert Burch theorem about minimal resolutions for $I$ 
(see \cite[Theorem 20.15ff]{E}). \par The result \eqref{genii} appears to be shown for at least $e=2$ in \cite{Mac1}. The general case follows when $\cha K=0$ or $\cha K=p>n$
from considering normal bases (\cite{Br,I1}), or in all characteristics from considering ``weak normal'' bases 
\cite[Theorem 4.3]{I1}.
Underlying the numerical result when $e=2$ is that a \emph{graded}
CI such as
$\,C=R/(x^a,y^b),a\le b$ has Hilbert function
\begin{equation*}
H(C)=(1,2,\ldots ,a,a,\ldots a,a-1,\ldots ,1).
\end{equation*}
When $\mathcal A$ is CI of codimension two then ${\mathcal A}^*$ has a unique filtration by graded
modules whose successive quotients are shifted CI's \cite{I3}.
\end{proof}
\begin{remark}\label{cirem} When $H(\mathcal A)$ satisfies $h_{i-1}-h_i\le
1$ for $i\ge \nu$, then
$P(H)$ has decreasing parts that differ pairwise by at least two. For example, when $H=(1,2,3,4,3,3,2,1), P(H)=(8,6,4,1)$
\end{remark}\par\noindent
The following is the main result of \cite{KO}. Recall that $B=J_P$, the Jordan nilpotent matrix of partition $P$. Recall that $K$ is algebraically closed.
\begin{theorem}(T. Ko{\v{s}}ir and P. Oblak)\label{Gorensteinthm}
Let $A$ be generic in $\mathcal N_B$. Then  $K[A,B]$ is Gorenstein. When $K $ is algebraically closed and $\cha K=0$ or $\cha K>n$ then $Q(P)$ is stable.
\end{theorem} 
\begin{proof}[Proof idea] Their key step is to extend V. Baranovsky's result that $A$ generic in $\mathcal N_B$ implies $K[A,B]$ is
cyclic \cite[Lemma 3]{Bar}, to show that $K[A,B]$ is also cocyclic (Gorenstein). Since height two Gorenstein Artinian algebras are CI (\cite{Mac3}), it follows that
for $A$ generic in $\mathcal N_B$, that $P(H)$ for $H=H(K[A,B])$ has parts that differ pairwise by at least two. They conclude using Theorem~\ref{decreasebthm} and Theorem~\ref{stablecor} that $Q(P)=P(H)$ and is stable.
\end{proof}\par\noindent
\begin{remark} The Ko{\v{s}}ir-Oblak theorem gives an alternative route to the conclusion of the first step in J. Brian\c{c}on's proof of his irreducibility theorem,
in which he ``vertically'' deforms an ideal to a complete intersection (\cite{Br}, see also \cite[p. 81ff]{I1}). Conversely, J. Brian\c{c}on's proof joined with V. Baranovsky's cyclicity result appears to give,
for $K$ algebraically closed of $\cha K=0 $ or $\cha K = p>n$, an alternative, if indirect, approach to the cocylicity step of the Oblak-Ko{\v{s}}ir result, since
\begin{enumerate}[a.] 
\item the vertical deformation preserves the Jordan partition of $m_x$ \cite[(5.2) and \S 5AI.3]{I1};
\item a deformation of a complete intersection remains a CI, and $\mathcal N_B$ is irreducible (Lemma \ref{nilpBirredlem}).
\end{enumerate}
\par J. Brian\c{c}on's proof of the irreducibility of $H^{[n]}$ requires a specific step to deform the CI $(xy, x^p+y^q)$ to an order one ideal. It would be interesting
to know the order $\nu_{Q(P)}$ of $H(Q(P))$ (the diagonal lengths of $Q(P)$) in terms of $P$. This order of $H(Q)$ is just the largest $\nu$ such that 
$Q_i\ge \nu+1-i$ for each $i, 1\le i\le \nu$.  
\end{remark}\par\noindent
{\bf Question.}  What is the
closure of $\mathcal U_{\nu ,n}$ in $H(n,K)$?  (See Corollary \ref{closureUnucor}).
\begin{ack} We thank P.~Oblak and
T. Ko{\v{s}}ir for communicating their result that $Q(P)$ is stable \cite{KO}. We thank J.~Weyman
for helpful discussions.  We thank F. Bergeron and A.~Lauve of UQAM, and K. Dalili 
and S. Faridi of Dalhousie, who organized the January 2007 Fields Mini-Conference, ``Algebraic
Combinatorics Meets Inverse Systems'' at UQAM; this provided a congenial atmosphere, and the opportunity
to present and develop results. At the conference we had stimulating discussions with B.~Sethuraman, who 
also communicated to us the results \cite{Oblak,Oblak2} of P. Oblak. We benefited from the comments of 
G. McNinch, and include an example he provided. We are grateful to the referee for a careful reading and many helpful suggestions
that greatly improved, and sometimes corrected, our exposition.
\end{ack}

\noindent\bigskip
Author e-mail:\smallskip  \par robasili@alice.it\par
a.iarrrobino@neu.edu
\end{document}